\documentclass[12pt,a4paper,reqno]{amsart}

\usepackage{amsmath,amssymb}
\usepackage{amsthm}
\usepackage{enumitem}
\usepackage[toc]{appendix}
\usepackage{multicol}
\usepackage{comment}
\usepackage{hyperref}
\usepackage{url}
\usepackage{graphicx}
\usepackage{mathtools}
\usepackage[capitalize]{cleveref}

\allowdisplaybreaks
\newcommand{\ignore}[1]{}

\theoremstyle{definition}
\newtheorem{definition}{Definition}

\newtheorem{remark}[definition]{Remark}
\newtheorem{notation}[definition]{Notation}

\theoremstyle{plain}
\newtheorem{proposition}[definition]{Proposition}
\newtheorem{lemma}[definition]{Lemma}
\newtheorem{theorem}[definition]{Theorem}
\newtheorem{conjecture}{Conjecture}

\newtheorem{corollary}[definition]{Corollary}
\newcommand*{\claimproofname}{Proof}

\newcommand{\F}{\mathbb F}
\newcommand{\cL}{\mathcal L}
\newcommand{\C}{\mathcal C}
\newcommand{\ZZ}{\mathbb Z}
\newcommand{\cI}{\mathcal{I}}
\newcommand{\cB}{\mathcal{B}}

\newcommand{\cM}{\mathcal M}
\newcommand{\cO}{\mathcal O}
\newcommand{\cZ}{\mathcal Z}
\newcommand{\cV}{\mathcal V}
\newcommand{\cH}{\mathcal H}
\newcommand{\NN}{\mathbb N}
\newcommand{\bcM}{\mathbf{\cM}}
\newcommand{\cP}{\mathcal P}
\newcommand{\cX}{\mathcal{X}}
\newcommand{\cS}{\mathcal{S}}

\DeclareMathOperator{\rowsp}{rowsp}
\DeclareMathOperator{\rk}{rk}
\DeclareMathOperator{\cl}{cl}

\newcommand{\qbinom}[2]{\genfrac[]{0pt}{0}{#1}{#2}}

\usepackage{xcolor}

\title{Counting $q$-matroids}
\author{Benjamin Jany}\thanks{\texttt{b.jany@tue.nl},  Department of Mathematics and Computer Science, Eindhoven University of Technology, The Netherlands}
\author{
Relinde Jurrius}\thanks{\texttt{RPMJ.Jurrius@mindef.nl},Faculty of Military Sciences, Netherlands Defence Academy, The Netherlands}
\author{
Rudi Pendavingh}\thanks{\texttt{r.a.pendavingh@tue.nl},  Department of Mathematics and Computer Science, Eindhoven University of Technology, The Netherlands}
\date{}

\begin{document}

\begin{abstract}
$q$-Matroids, a $q$-analogue of classical matroids have attracted a lot of attention over the last decade, yet their enumeration remains largely unexplored. In this paper, we study the number of $q$-matroids, paving and sparse-paving $q$-matroids defined on a fixed ground space and with prescribed rank. We derive new lower bounds using constructions from constant-dimension codes and improve existing estimates. On the upper bound side, we develop two approaches: a combinatorial method based on controlling the number of dependent hyperplanes for paving $q$-matroids, and an entropy-based counting argument applicable to classes of $q$-matroids closed under contraction. These techniques yield explicit upper bounds on the logarithmic number of $q$-matroids with fixed rank and ground space. Finally, we analyze the asymptotic behavior of these bounds, and identify gaps between lower and upper estimates, leading to conjectures on the true asymptotic growth.
\end{abstract}

\maketitle

\section{Introduction}
$q$-Matroids, a $q$-analogue of matroids, first introduced in \cite{crap}, were rediscovered in \cite{JP2018} establishing a connection with rank-metric codes. A $q$-matroid can be defined via a submodular, non-decreasing and bounded rank function function on the collection of subspaces of a vector space. 

Since their reintroduction, the theory of $q$-matroids has blossomed, both because of its connection to coding theory, but also because of its interesting similarities and contrasts with classical matroid theory. 
Much of recent work on $q$-matroids has focused on cryptomorphisms \cite{AB24,BCJ22}, combinatorial structure \cite{CJ24,ConcaJanyRavagnani2026,GJ22, GJ24decomp,GluesingLuerssenJany2023}, and the relation to rank-metric code and/or finite geometry \cite{AJNZ24, BIJS23,GJ25Rep,Jany2023Projectivization}. 

Despite much progress in the topics mentioned previously, there has been limited work on the enumeration of $q$-matroids. Understanding the number of $q$-matroids is a fundamental step toward characterizing their typical structure, much as in classical matroid theory, where enumeration has played a key role in revealing that most matroids exhibit highly non-representable and extremal behavior.
In \cite{Degen24}, the authors focus on the density of representable $q$-matroids (i.e. $q$-matroids that arise from $\F_{q^m}$-linear rank-metric codes). They more precisely show that almost all $q$-matroids are not representable. 
  
In this paper, we focus on a different enumeration problem: what is the number of $q$-matroids over a fixed ground space and of given rank? In order to give an asymptotic answer to this question, we first consider the following, more restrictive, problem: what is the number of paving $q$-matroids of given rank and with fixed ground space?
Both problems are $q$-analogues of enumeration in matroid theory which have been intensively studied  \cite{BPvdP2014,Knuth1974,KwanSahSawhney2024, PvdP2015, PvdP2016,PvdP2017}.
We establish upper bounds on the number of paving $q$-matroids and $q$-matroids with fixed ground space and fixed rank. 
We then consider asymptotic behaviors of those bounds, for example when the field size goes to infinity, when the dimension of the ground space goes to infinity and the rank is half on the dimension of the ground space and the former goes to infinity. 

Similar techniques as for classical matroids are used, which leads us to make a rather direct $q$-analogue of the basic counting results from the classical case \cite{BPvdP2014,PvdP2017}. We expect more subtle counting results, leading to better upper- and lower bounds \cite{PvdP2016,HPvdP2022}, to have a direct $q$-analogue as well. We formulate two conjectures about the asymptotic count.

The article is structured as follows: in Section \ref{sec:preliminaries} we set notation and recall definitions related to $q$-matroids; in Section \ref{sec:lowerbound} we consider lower bound on the number of $q$-matroids and slightly improve on existent bounds; in Section \ref{sec:upperbpunds} we establish an upper bound on the number of paving $q$-matroid by bounding the number of dependent hyperplanes, and an upper bound on the number of $q$-matroids using an entropy counting method; in Section \ref{sec:asymptoticcomparisson} we consider asymptotic counts and give a count for the number of $q$-matroids with fixed ground space and fixed rank over large finite fields.

\medskip

\section{Preliminaries}\label{sec:preliminaries}

\subsection{Fields, spaces and subspaces}
\begin{notation} We write $\F_q$ for the {\em Galois field} (or {\em finite field}) of cardinality $q$.
\end{notation}

\begin{notation}
    If $V$ is a vector space, then $U\leq V$ denotes that $U$ is a linear subspace of $V$. We write
    $$\cL(V):=\{U: U\leq V\}$$
    for the set of all linear subspaces of $V$, and 
    $$\begin{bmatrix} V \\ k \end{bmatrix}_q:=\{U: U\leq V, \dim(U)=k\}$$
    for the set of subspaces of dimension $k$.
\end{notation}


For $n \in \NN$ and a prime power $q$, let $$[n]_q = \frac{q^n-1}{q-1}\text{ and }[n]_q! = [n]_q [n-1]_q \cdots [2]_q[1]_q.$$ The \emph{Gaussian binomial coefficients \(\bigl[\!\begin{smallmatrix} n \\ k \end{smallmatrix}\!\bigr]_q\)}, for \(n,k \in \mathbb{N}\setminus \{0\}\), \(n \geq k\) and prime power \(q \geq 2\), are defined as:
\[\begin{bmatrix} n \\ k \end{bmatrix}_q:= \frac{[n]_q!}{[k]_q! [n-k]_q!}=\frac{\left(q^n-1\right) \cdots\left(q^{n-k+1}-1\right)}{\left(q^k-1\right)\cdots(q-1)}.
\]
Furthermore, if \(k=0\) then \(\bigl[\!\begin{smallmatrix} n \\ k \end{smallmatrix}\!\bigr]_q =1\), and   \(\bigl[\!\begin{smallmatrix} n \\ k \end{smallmatrix}\!\bigr]_q = 0\) if \(k<0\) or \(n<k\). Note also that \(\bigl[\!\begin{smallmatrix} n \\ 1 \end{smallmatrix}\!\bigr]_q = [n]_q\) for all $n \in \NN$. 

It is well known that if $V$ is any vector space over $\F_q$ of dimension $n$, then 
$$\left|\qbinom{V}{k}\right|= \qbinom{n}{k}_q.$$

\subsection{$q$-Matroids}

\begin{definition}
    A $q$-matroid is a pair $M = (E , \rho)$, where $E$ is a vector space, and $\rho : \cL(E) \rightarrow \ZZ_{\geq 0}$, so that for all $V, W \in \cL(E)$ 
    \begin{itemize}
        \item[(R1)] $0 \leq \rho(V) \leq \dim V$,
        \item[(R2)] if $V \leq W$ then $\rho(V) \leq \rho(W)$,
        \item[(R3)] $\rho(V+W) + \rho(V \cap W) \leq \rho(V) +\rho(W)$.
    \end{itemize}
    The {\em ground space} of $M$ is $E(M):=E$, and the {\em rank function} of $M$ is $\rho_M:=\rho$. The {\em size} of $M$ is $\dim(E)$ and the {\em rank} of $M$ is $\rho(M) :=\rho(E)$ (also often denoted by $r$).
\end{definition}

Similarly to classical matroids, one can define several classes of spaces within a $q$-matroid. We recall some of them here.
Let $M = (E, \rho)$ be a $q$-matroid. A subspace $V \leq E$ is \textbf{independent} if $\rho(V)  = \dim(V)$, otherwise $V$ is \textbf{dependent}. $V$ is a \textbf{basis} if $V$ is independent and $\dim(V)=\rho(E)$. Furthermore, $V$ is a \textbf{circuit} if it is dependent and if any proper subspace of $V$ is independent.  Moreover, $V$ is an \textbf{open space} if it is the $\{0\}$-space or a sum of circuits. 

The \textbf{closure} of $V$ is defined as $\cl(V) := \{w \in E \, : \, \rho(V + \langle w \rangle) = \rho(V)\}$ and $V$ a \textbf{flat} if $\cl(V) = V$. Additionally, a flat $V$ is called a \textbf{hyperplane} if $\rho(V) = \rho(E)-1$. Finally, $V$ is a \textbf{cyclic flat} if $V$ is both an open space and a flat. 

\begin{notation}
    Let $M = (E, \rho)$ be a $q$-matroid. Then:
    \begin{itemize}
        \item $\cI_{M} := \{I \leq E \, : \, I \text{ is independent}\}$, 
        \item $\cB_{M} := \{B \leq E  \, : \, B \text{ is a basis}\}$,

        \item $\cH_{M} := \{H \leq E \, : H \text{ is a hyperplane}\},$
        \item $\cZ_{M} := \{Z \leq E  \, : \, Z \text{ is a cyclic flat}\}$.
    \end{itemize}
    The subscript is omitted if the $q$-matroid $M$ is clear from the context.
\end{notation}
Each $q$-matroid $M$ is uniquely determined by its ground space $E(M)$ and either its collection of independent sets $\cI_{M}$, or its bases $\cB_M$, or its hyperplanes $\cH_M$ \cite{AB24, BCJ22}.

\begin{definition}
    A $q$-matroid $\cM$ of rank $r$ is \emph{paving} if every circuit has dimension at least $r$. Furthermore a $q$-matroid is \emph{sparse-paving} if it is paving and any two circuits of dimension $r$ intersect in a subspace of dimension at most $r-2$.
\end{definition}

The notion of paving $q$-matroid is analogous to the notion of paving matroid for classical matroids. 

It was shown in \cite{PvdP2016, HPvdP2022} that the logarithmic number of (sparse) paving matroids and the logarithmic number of matroids have the same asymptotic rate. Also, upper bounds on the number of paving matroids are a stepping stone to bounds on the number of general matroids. Paving $q$-matroids will play a similar role in this paper.

\begin{definition} Let $M=(E, \rho)$ be a $q$-matroid and let $\phi:E\rightarrow F$ be a vector space isomorphism. Then $\phi(M):= (F, \rho\circ(\phi^{-1}))$. We say that two $q$-matroids $M, M'$ are {\em isomorphic} if there is a vector space isomorphism $\phi: E(M)\rightarrow E(M')$ so that $M'=\phi(M)$.
\end{definition}

When counting the number of $q$-matroids, we will consider $M$ and $M'$ as distinct when $M\neq M'$, even if $M$ and $M'$ are isomorphic. To then avoid overcounting the number of $q$-matroids of a given size and rank, we to restrict our attention to $q$-matroids with a fixed ground space of the form $\F_q^n$ for some $n$. Hence the following notational choices.

\begin{notation}\label{not:numberqmat}
Let $\cM_q$, $\cP_q$, $\cS_q$ denote respectively the classes of $q$-matroids, paving $q$-matroids and sparse-paving $q$-matroids.
Furthermore for $\cX \in \{\cM, \cP, \cS\}$ and $x \in \{m,p,s\}$ (where $\cX$ and $x$ are the same letter with different font style) let 
$$\cX_q(n) := \{M \, \in \cX_q : \, E(M) = \F_q^n\} \quad \text{and} \quad x_q(n) := |\cX_q(n)|,$$ 
 $$\cX_q(r,n) := \{ M \in \cX_q(n)\, : \rho(M) = r\} \quad \text{and} \quad x_q(r,n) := |\cX_q(r,n)|$$

More precisely, $\cX_q$ corresponds to a class of $q$-matroids and $x_q$ to the number of $q$-matroids in that class.

  \end{notation}
As $\cS_q\subseteq \cP_q\subseteq \cM_q$, we have
$$s_q(r,n)\leq p_q(r,n)\leq m_q(r,n)\qquad\text{for all }0\leq r\leq n.$$

To finish the section, we include the definition of restriction and contraction for $q$-matroids, first introduced in \cite{JP2018}. 

\begin{definition}
Let $M=(E,\rho)$ be a $q$-matroid and let $T\leq E$. For every subspace $A\subseteq T$, define $\rho_{M|T}(A)=\rho(A)$. Then $M|T:=(T,\rho_{M|T})$ is called the restriction of $M$ to $T$.
\end{definition}

\begin{definition}
Let $M=(E,\rho)$ be a $q$-matroid and let $T\leq E$. We define the map $\rho_{M/T}:\mathcal{L}(E/T)\to\mathbb{Z}$ via 
$$\rho_{M/T}(A/T)=\rho(A)-\rho(T)\qquad\text{ for all }T\leq A\leq  E.$$ 
Then $M/T:=(E/T,\rho_{M/T})$ is called the contraction of $M$ by $T$.
\end{definition}

It is well-known  that $M|T$ and $M/T$ are $q$-matroids.


\section{Lower bound}\label{sec:lowerbound}
A first lower bound on the number of (sparse paving) $q$-matroids with ground space $\F_q^n$ was established in \cite{Degen24}. It is based on a construction of \emph{constant dimension codes} (CDC). We present it here for completeness. 

\begin{definition}
    A non-empty subset $\C \subseteq \cL(\F_q^n)$ is a called a subspace code. Its minimum subspace distance is given by 
    $$d_S(\C) := \min \{ d_S(V,W) \, : \, V, W \in \C, V \neq W\},$$
    where $d_S(V,W) := \dim V + \dim W - 2 \dim( V \cap W)$. Furthermore $\C$ is called a constant dimension code (CDC) of dimension $k$ if $\dim V = k$ for all $V \in \C$. Let $A_q(n,d,k)$ be the maximal cardinality of a CDC $\C \subseteq \cL(\F_q^n)$ of dimension $k$ and minimum distance $d_S(\C) \geq d$. 
\end{definition}

A lower bound on $A_q(n,d,k)$ was established in \cite[Section IV]{Silva08}. The bound arises via the lifting construction of a CDC. First, recall the rank distance is defined as $d_R: \F_q^{n \times m} \times \F_q^{n \times m}$ where $d_R(X,Y) = \rk(X-Y)$. Given $\C \subseteq \F_q^{n \times m}$, its minimum rank distance is $$d_R(\C) := \min \{d_R(X,Y) \, : \, X, Y \in \C, X \neq Y \}.$$

Let $\cI : \F_q^{k \times (n-k)} \rightarrow \cL(\F_q^{n })$ where $X \mapsto \rowsp( [I_k \, | \, X] )$. The subspace $\cI(X)$ is called the lifting of $X$. Furthermore if $\C \subseteq \F_q^{k \times (n-k)}$ then $\cI(\C) := \{\cI(X) \, : \, X \in \C\}$ is called the lifting of $\C$. 

It is shown in \cite[Prop. 4]{Silva08} that for all $X, Y \in \F_q^{k \times (n-k)}$ we have $d_S(\cI(X), \cI(Y)) = 2d_R(X, Y)$. Furthermore, for $\C \subseteq \F_q^{k \times (n-k)}$, we have $d_S(\cI(\C)) = 2d_R(\C)$. Note that $\cI(\C)$ is a CDC of cardinality $|\C|$.
Hence from the rank-metric Singleton bound which states that $|\C| \leq q^{\max \{k,(n-k)\}(\min \{k, (n-k)\} -d+1)}$ for $\C \subseteq \F_q^{n\times m}$ and the fact that \emph{MRD codes} (codes that achieve the Singleton bound) exist for all choice of parameters $n,k,q,d$  (see \cite{Gabidulin1985} for a construction), we get the existence of a CDC with parameters $n,2d,k,q$ of cardinality $q^{(n-k)(k-d+1)}$. In the following statement, we also include an upper bound \cite{Kotter08} for  $A_q(n,2d,k)$. 

\begin{proposition}\label{prop:boundAq}\cite{Kotter08} 
    For $2k \leq n$ and $d \geq 2$, we have
    $$ q^{(n-k)(k - d +1)} \leq  A_q(n,2d,k) < 4q^{(n-k)(k-d+1)}$$
\end{proposition}

Now note that it is possible to derive a $q$-matroid from a CDC of minimum distance at least $4$. 

\begin{proposition}\cite{GJ22}
    Let $\C \subseteq \cL(\F_q^n)$ a CDC of dimension $k$ such that $\dim (V \cap W) \leq k-2$ for all $V, W \in \C$ (i.e. $d_S(\C) \geq 4$). Define 

    $$\rho: \cL(\F_q^n) \rightarrow \ZZ_{\geq 0}, \, \, V \mapsto \begin{cases}
        k-1 & \text{if }V \in \C \\
        \min\{ \dim V, k\} & \text{otherwise}.
    \end{cases}$$
Then $(\F_q^n, \rho)$ is a sparse paving $q$-matroid of rank $k$, whose circuits of rank $k-1$ are the elements of $\C$.
\end{proposition}

Furthermore, note that a CDC of dimension $k$ and minimum distance at least $4$ is an independent set of the Grassmann graph $J_q(n,k)$, therefore $A_q(n,4,k)$ is the independence number of $J_q(n,k)$. Recall, \emph{Grassmann graph} (or \emph{q-Johnson graph) \(J_q(n,k)\)} has as vertices the \(k\)-subspaces of $\F_q^n$,
where two vertices are adjacent if the corresponding \(k\)-subspaces intersect in a \((k-1)\)-space. 
Using the previously stated facts, one gets the following lower bound on the number of $q$-matroids of rank $r$ over $\F_q^n$, which is in fact a lower bound on the number of paving $q$-matroids of rank $r$ over $\F_q^n$. 

\begin{theorem}\cite[Thm. 3.5]{Degen24}\label{thm:lowerboundrank}
    Let $n \geq 4$ and $2 \leq r \leq \lfloor \frac{n}{2} \rfloor$. Then 

    \begin{equation}\label{eq:lowerboundrank}
    2^{q^{(n-r)(r-1)}} \leq s_q(r,n). 
\end{equation}
\end{theorem}

We can slightly improve this lower bound on $s_q(r,n)$ as described below.

\begin{proposition}\label{prop:lowerboundnumpaving}
    Let $n \geq 4$ and $2 \leq r \leq \left \lfloor \frac{n}{2} \right \rfloor$. Then 
    $$q^{3(n-r)}2^{q^{(n-r)(r-1)}} \leq s_q(r,n).$$
\end{proposition}

\begin{proof}
    Let $\C \leq \F_q^{r \times n-r}$ be an MRD code of minimum distance $4$. Note that any coset of $\C$ is also an MRD code of minimum distance $4$. Hence $\cI(\C')$, where $\C'$ is a coset of $\C$, is a subspace code of minimum distance $2$. Since $\cI$ is injective, $\cI(\C') \cap \cI(\C'') = \emptyset$ for any different coset $\C'$, $\C''$ of $\C$. Therefore there exist $q^{3(n-r)}$ independent sets of the $q$-Jonhson graph of cardinality $q^{(n-r)(r-1)}$. As each subset of the independent set induces a distinct $q$-matroid of rank $r$ over $\F_q^n$, the result follows.  
\end{proof}

\medskip

\section{Upper bounds}\label{sec:upperbpunds}

In this section, we derive upper bounds on the number of $q$-matroids with fixed rank and fixed ground space. We first derive a bound on the number of paving $q$-matroids through a study of their dependent hyperplanes. We then derive a bound of the number of $q$-matroids by using an entropy counting method. 

\subsection{Paving $q$-matroids}

We will denote the collection of dependent hyperplanes of a $q$-matroid $M$ by
$$\cH^*_M:=\cH_M\setminus \cI_M.$$

\begin{lemma} Let $M=(E,\rho)$ be a paving $q$-matroid. 
Then $$\cH_M=\cH^*_M\cup\left\{V\leq \qbinom{E}{r-1}_q: V\not\leq H\text{ for all }H\in \cH^*_M\right\}.$$
It follows that each $M\in \cP_q(r,n)$ is uniquely determined by $\cH^*_M$.
\end{lemma}
\begin{proof}
    $\cH_M\subseteq\{\ldots\}$: If $V\in \cH_M\setminus\cH^*_M$, then $V$ is  an independent hyperplane and, by definition,
    $\dim(V)=\rho(V)=r-1$. If $H$ is any hyperplane so that $V\leq H$, then $V=H$ as $V$ and $H$ are both hyperplanes, but then $H=V\in\cI_M$ and hence $H\not\in \cH^*_M$.

    $\cH_M\supseteq\{\ldots\}$: Clearly $\cH_M\supseteq \cH^*_M$. Suppose that $V$ is an $(r-1)$-space. By our assumption that $M$ is paving, we have $\rho(V)=\dim(V)=r-1$. Then $H:=\cl(V)$ is a hyperplane of $M$ so that $V\leq H$. If $V$ is not a hyperplane, then $V\neq H$ and hence $\rho(H)=\rho(V)=\dim(V)<\dim(H)$. Then $H\in \cH^*_M$.

    Finally, each $q$-matroid is uniquely determined by its ground space and collection of hyperplanes. If $M\in \cP_q(r,n)$, then the ground space of $M$ is $E=\F_q^n$ and the rank of $M$ is $r$, and by the above equation, the set of hyperplanes  $\cH_M$ is determined by $E, r$, and the dependent hyperplanes $\cH^*_M$.
\end{proof}

The dependent hyperplanes of a paving $q$-matroid are precisely its cyclic flats excluding $\{0,E\}$. Although this fact will not be used in the remainder of the paper, we still believe it is worth stating and proving, given the importance of cyclic flats in the study of $q$-matroids (see, for example, \cite{AB24, GJ24decomp}).

\begin{proposition}\label{prop:cylicflatpav}

    If a $q$-matroid $M = (E, \rho)$ of rank $r$ is paving then 
    \begin{align*}\cZ_{M} \setminus \{0, E\} &= \{ \cl(C) \, : \, C \in \C_{M} \text{ and } \rho(C) = r-1\}= \cH^*_M
    \end{align*}
\end{proposition}

\begin{proof}
    Let $Z$ be a cyclic flat of $M$ that is neither $\cl(0)$ nor $E$. Since $Z \neq E$ then $\rho(Z) < \rho(E) = r$. Furthermore, $Z \neq 0$ hence, by definition, it is a non-trivial sum of circuits. Therefore, there exists a circuit $C \le  Z$. Since $M$ is paving, $\rho(C) \geq r-1$. Combining the previous facts, we get $r-1 \leq \rho(C) \leq \rho(Z) \leq r-1$, hence equality holds throughout. Finally since $Z$ is a flat both containing $C$ and of same rank as $C$ it must be its closure. The reverse inclusion follows from \cite[Lemma 4.1 b)]{GJ24decomp}.

   For the second equality. Let $C \in \C_M$ such that $\rho(C) = r-1$. Then $\cl(C)$ is a flat of rank $r-1$, hence a hyperplane, as well as dependent (since it contains a circuit). For the reverse inclusion, let $H \in \cH^*_M$. Since $H$ is a dependent hyperplane, $\rho(H) = r-1$ and it must contain a circuit, $C$. Furthermore, $M$ is paving thus $r-1 \leq \rho(C)$. By monotonicity of the rank function $r-1 \leq \rho(C) \leq \rho(H) \leq r-1$, which implies equality throughout. Since $H$ is a flat containing $C$ and of same rank as $C$, it must be that $\cl(C) = H$.
\end{proof}


\begin{remark}
 Note that in a paving $q$-matroid the space $\{0\}$ is always a cyclic flat. However, one can show, the groundspace $E$ is a cyclic flat if and only if $\sum_{H \in \cH^*_M} H = E$. Since cyclic flats together with their rank values fully determine a $q$-matroid, this line of reasoning can also be used to show that dependent hyperplanes fully determine the $q$-matroid. 
\end{remark}

We will next describe how the set of dependent hyperplanes $\cH^*$ of a paving matroid of rank $r$ may be reconstructed from a sufficiently rich collection of dependent $r$-spaces
$$\cV\subseteq \bigcup_{H\in \cH^*}\qbinom{H}{r}.$$
To estimate the number of paving $q$-matroids, we will argue that such a  collection $\cV$ need not be large, and that hence there are no more paving $q$-matroids than sets of $r$-spaces of bounded size.

The following lemma is key.
\begin{lemma}\label{lem:deprspaceclos}
    Let $M$ be a paving $q$-matroid of rank $r$. If $V$ and $W$ are spaces of rank $r-1$ in $M$ and $\dim(V \cap W) = r-1$, then $\cl(V) = \cl(W)$.
\end{lemma}

\begin{proof}
Since $M$ is paving, $V\cap W$ is independent and hence $\rho(V) =\rho(W) = r-1 = \rho(V \cap W)$. Then $\cl(V) = \cl(V \cap W) = \cl(W)$.
\end{proof}
We will construct $\cV$ by selecting a suitable subset of $r$-sets for each dependent hyperplane $H$.
\begin{lemma} \label{lem:H_encoding}Let $H$ be a linear space of dimension $t$ and let $r\leq t$. Then there exist a collection
$$\cV(H)=\{V_0, \ldots, V_{t-r}\}\subseteq \qbinom{H}{r}$$
so that $\dim(V_{i-1}\cap V_{i})=r-1$ for $i=1,\ldots, t-r$ and $H=V_0+\cdots +V_{t-r}$.
\end{lemma}
\begin{proof}
    Pick any tower of subspaces 
    $U_0\leq U_1\leq \cdots\leq U_{t-r}=H$ 
    so that $\dim(U_i)=r+i$. Let $V_0=U_0$, and for $i=1,\ldots, t-r$, let $V_i\leq U_i$ be any $r$-space so that $\dim(V_i\cap V_{i-1})=r-1$ and $V_i\not\leq U_{i-1}$. Then $\dim(U_{i-1}+V_i)>\dim(U_{i-1})$ and $U_{i-1}+V_i\leq U_i$, so that 
    $U_{i-1}+V_i=U_i$. By induction $V_0+\cdots+ V_i=U_i$ for $i=0,\ldots, t-r$, and hence $V_0+\cdots+ V_{t-r}=U_{t-r}=H$ as required.
\end{proof}
For each $M\in \cP_q(r,n)$, we put
$$\cV(M):=\bigcup\{\cV(H): H\in \cH^*_M\}$$
where each $\cV(H)$ is any set of $r$-spaces $\cV$ as in Lemma \ref{lem:H_encoding}.

\begin{lemma}\label{lem:rec}
    Let $M\in \cP_q(r,n)$, and consider the undirected graph
    $$G:=(\cV(M), X)$$
    with set of edges $X:=\{\{V, W\}: \dim(V\cap W)=r-1\}.$ Then 
    $$\cH^*_M=\left\{\sum_{V\in \cO} V: \cO\text{ a connected component of }G\right\}.$$
\end{lemma}
\begin{proof} By construction, each $V\in \cV(M)$ is a dependent $r$-space of $M$, and as $M$ is paving the rank of such $V$ is $r-1$.
If $\{V, W\}$ is an edge of $G$, then $\dim(V\cap W)=r-1$, so that then $\cl(V)=\cl(W)$ by Lemma \ref{lem:deprspaceclos}. 
It follows that for each connected component $\cO$ of $G$ there is a space $H_{\cO}$ so that $\cl(V)=H_{\cO}$ for each $V\in \cO$. 
As $V\leq \cl(V)=H_{\cO}$ for each $V\in \cO$, we have $\sum_{V\in \cO} V \leq H_{\cO}$. 

 Consider an $H\in \cH_M^*$. Then $\cV(H)=\{V_0, \ldots, V_{k}\}\subseteq \cV(M)$, where $\dim(V_{i-1}\cap V_{i})=r-1$ for $i=1,\ldots, k$ and $H=V_0+\cdots+V_{k}$. Then each $\{V_i, V_{i+1}\}$ is an edge of $G$. It follows that if  $\cO$ is a component of $G$ such that $\cV(H)\cap \cO\neq \emptyset$, then  $\cV(H)\subseteq \cO$, and hence 
    $$H=V_0+\cdots +V_k\leq \sum_{V\in \cO} V\leq H_{\cO}=\cl(V_0)\leq H$$
so that $H=H_{\cO}$. Then $\cO=\cV(H)$, for if $H'\in \cH_M^*$  has $\cV(H)\cap \cO\neq \emptyset$ as well, then $H'=H_{\cO}=H$. Hence
    $$\{\cV(H): H\in \cH^*_M\}=\{\cO: \cO\text{ a connected component of }G\}.$$
The lemma follows using that $H=\sum_{V\in \cV(H)} V$ for each $H\in \cH^*_M$.
\end{proof}

\begin{proposition}\label{prop:reconstructcyclic}
    Each $M\in \cP_q(r,n)$ is uniquely determined by $\cV(M)$.
\end{proposition}
\begin{proof} Let $M\in \cP_q(r,n)$. In Lemma \ref{lem:rec}, the graph $G$ depends only on $\cV(M)$ and the rank $r$, and hence $\cH^*_M$ is determined by $r, \cV(M)$. By Lemma \ref{lem:H_encoding}, $M$ is uniquely determined by $\cH^*_M$.
\end{proof}

We next determine an upper bound on the cardinality of $\cV(M)$. 

\begin{lemma}\label{lem:vmupperbound}
   Let $M\in \cP_q(r,n)$. Then 
    $$|\cV(M)| \leq \frac{1}{[n-r+1]_q}\qbinom{n}{r}_q.$$
\end{lemma}

\begin{proof} Let $h_t := |\{H \in \cH^*_M \, : \dim(H) = t\}|$. For each 
    $H \in \cH^*_{M}$ of dimension $t=\dim(H)$ we have $|\cV(H)| = t -r +1$ as $\cV(H)$ is chosen according to Lemma \ref{lem:H_encoding}.
    Hence
    $$|\cV(M)|\leq  \sum_{H\in \cH^*_M}|\cV(H)|= \sum_{t=r}^{n-1} (t-r+1)h_t.$$
    By Lemma \ref{lem:deprspaceclos}, each $(r-1)$-subspace of $E(M)$ is contained in at most one $H\in \cH^*_{M} $. Hence 
    $$\sum_{t=r}^{n-1} \qbinom{t}{r-1}_q h_t = \sum_{H\in \cH^*_M} \left|\qbinom{H}{r-1}\right|\leq \left|\qbinom{E(M)}{r-1}\right|=\qbinom{n}{r-1}_q.$$
    It follows that $|\cV(M)|$ is at most the optimal value of the LP
    \begin{align*}\max ~~&\sum_{t=r}^{n-1} (t-r+1)h_t \\ &\sum_{t=r}^{n-1} \qbinom{t}{r-1}_qh_t \leq \qbinom{n}{r-1}_q\\
    &h_t \geq 0 \text{ for } t=r, \ldots,  n-1.
    \end{align*}
By LP duality, this equals
    \begin{align*}\min ~~&y\qbinom{n}{r-1}_q\\
    &y\qbinom{t}{r-1}_q \geq t-r+1 &\text{ for } t=r, \ldots,  n-1\\
    &y\geq 0.
    \end{align*}
The dual optimal solution then is
$$y^*= \max\left\{(t-r+1)/\qbinom{t}{r-1}_q:t=r, \ldots,  n-1\right\} = 1/\qbinom{r}{r-1}_q$$ as the maximum is attained when $t = r$. Then the optimal value of the dual LP is $$y^*\qbinom{n}{r-1}_q=\frac{1}{[r]_q}\qbinom{n}{r-1}_q = \frac{1}{[n-r+1]_q}\qbinom{n}{r}_q,$$
as required.
\end{proof}

We can now derive an upper bound on the number of paving $q$-matroids.

\begin{theorem}\label{thm:nbrpaving}
    For all $0\leq r\leq n$, we have
    $$\log(p_q(r,n)) \leq \frac{1}{[n-r+1]_q}\qbinom{n}{r} \log(e[n-r+1]_q).$$
\end{theorem}

\begin{proof}
By Proposition \ref{prop:reconstructcyclic}, each $M\in \cP_q(r,n)$ is uniquely determined by the collection $\cV(M)\subseteq \qbinom{E(M)}{r}$, which by Lemma \ref {lem:vmupperbound} has cardinality at most 
$$\frac{1}{[n-r+1]_q}\qbinom{n}{r} =: t.$$
Using the standard bound $\sum_{i=0}^t \binom{s}{i} \leq (\frac{e \cdot s}{t} )^t$, we obtain
\begin{align*}
    p_q(r,n) \leq \sum_{i=0}^t \binom{\qbinom{n}{r}_q}{i} \leq (e[n-r+1]_q)^{t}.
\end{align*}
The statement then follows. 
\end{proof}

\subsection{An entropy counting method} 
In \cite{BPvdP2014}, an entropy counting technique was used to give an upper bound on the number of matroids of rank $r$ on a fixed ground set of cardinality $n$. We will develop the analogous technique for $q$-matroids in this section.

Our method will apply to any collection of $q$-matroids $\cX\subseteq \bcM_q$ which is {\em closed under contraction} in the following technical sense: for each $u, n\in \NN$ so that $0\leq u\leq n$ and  each subspace $U\leq \F_q^n$ of dimension $u$, there exists an isomorphism $\phi: \F_q^n/U\rightarrow \F_q^{n-u}$ so that $$\phi(M/U)\in \cX(n-u)\text{ for all }M\in \cX(n).$$ 
Note that $\cS_q$, $\cP_q$ and $\cM_q$ are closed under contraction in this sense.

As noted in the preliminaries, a $q$-matroid $M=(E,\rho)$ is uniquely determined by its ground space $E$ and set of bases $\mathcal{B}_{M}$. Indeed, we have
$$\rho(V)=\max\{\dim(B\cap V): B\in \cB_M\} \qquad\text{ for all }V\leq E.$$
As $E(M)=\F_q^n$ for each $M\in \cX(r,n)$, we have 
$$x(r,n):=|\cX(r,n)|=|\{\mathcal{B}_{M} : M \in \cX(r,n)\}|.$$
Given a subspace $U \leq \F_q^n$ of dimension $u$, the contraction minor $M/U$  is a $q$-matroid over $\F_q^n/U$. It is not difficult to verify that 
$$\cB_M/U:=\{B/U: B\in \mathcal{B}_{M}, U\leq B\}= \left\{\begin{array}{ll} \cB_{M/U}&\text{if }U\in \cI_M\\ 
\emptyset &\text{otherwise}\end{array}\right.$$
We will argue that if $\{\cB_M/U: M\in \cX(r,n)\}$ is a small set for each subspace $U$ of a fixed dimension, then 
$\cX(r,n)$ itself is relatively small. To this end, we will use the following result from \cite{CGFS1986}, which is an application of Shearer's Entropy Lemma.
\begin{theorem}[Shearer's Product Theorem]  Let $A$ be a finite set and let $A_1, \ldots A_m\subseteq A$ be such that each $a\in A$ occurs in at least $k$ of the $A_i$. Let $\mathcal{S}\subseteq 2^A$. Then
$$|\mathcal{S}|^k\leq \prod_{i=1}^m |\mathcal{S}_i|,$$
where $\mathcal{S}_i:=\{S\cap A_i: S\in\mathcal S\}$ for $i=1,\ldots,m$. 
\end{theorem}

\begin{theorem}\label{thm:entropybound}Let $\cX\subseteq \cM_q$ be a contraction-closed class of $q$-matroids and let $u,r,n\in \NN$ be such that $0\leq u\leq r\leq n$. Then
$$\log(x(r,n)+1)/\begin{bmatrix} n \\ r \end{bmatrix}_q\leq \log(x(r-u,n-u)+1)/\begin{bmatrix} n-u \\ r-u \end{bmatrix}_q.$$
\end{theorem}
\proof 
Let $E:=\F_q^n$ be the common ground space of the $q$-matroids in $\cX(r,n)$. Fix an enumeration  of the $u$-subspaces of $E$, say 
$$\{U_1,\ldots, U_m\}=\begin{bmatrix} E \\ u \end{bmatrix}_q, \qquad \text{where }m=\begin{bmatrix} n \\u  \end{bmatrix}_q.$$ 
Let $$A:=\begin{bmatrix} E \\ r \end{bmatrix}_q \text{ and }A_i:= \left\{B\in A: U_i\leq B\right\}\text{ for }i=1,\ldots, m.$$
Then for each $B\in A$, we have $$\{i: B\in A_i\}=\{i: U_i\leq B\}=\begin{bmatrix} r \\u  \end{bmatrix}_q=: k$$
as each $r$-subspace of $E$ contains exactly $k$ $u$-subspaces of $E$. Let 
$$\mathcal{S}:=\{\mathcal{B}_M: M\in \cX(r,n)\}\cup\{\emptyset\}.$$
Then $\mathcal{S}\subseteq 2^A$ and $ |\mathcal{S}|=x(r,n)+1$. We have 
$$\mathcal{S}_i:= \{S\cap A_i: S\in\mathcal S\}=\{\{B\in \mathcal{B}_M: U_i\leq B\}: M\in \cX(r,n)\}\cup\{\emptyset\}.$$
It follows that 
$$|\mathcal{S}_i|\leq |\{\mathcal{B}_{M/U_i}: M\in \cX(r,n), U_i\in \cI_M\}\cup\{\emptyset\}|$$
and as each $\mathcal{B}_{M/U_i}$ determines $M/U_i$, we have $$|\mathcal{S}_i|\leq |\{M/U_i: M\in \cX(r,n), U_i\in \cI_M\}\cup\{\emptyset\}|.$$ As $\cX$ is closed under contraction, there exists an 
isomorphism $$\phi_i:\F_q^n/U_i\rightarrow \F_q^{n-u}$$ for $i=1,\ldots,m$ so that $\phi_i(M/U_i)\in \cX(n-u)$ for all $M\in \cX(n)$. Being an isomorphism, each such $\phi_i$ induces a bijection between the $q$-matroids on $\F_q^n/U_i$ and the $q$-matroids on $\F_q^{n-u}$. Then
 $$|\mathcal{S}_i|\leq |\{\phi_i(M/U_i): M\in \cX(r,n), U_i\in \cI_M\}\cup\{\emptyset\}|.$$
Moreover, we have $\phi_i(M/U_i)\in \cX(r-u, n-u)$ for each $M\in \cX(r,n)$ so that $U_i\in \cI_M$, as $\dim(U_i)=u$. Hence
 $$|\mathcal{S}_i|\leq |\cX(r-u,n-u)\cup \{\emptyset\}|=x(r-u,n-u)+1.$$ 
By Shearers' Product Theorem, it follows that 
$$(x(r,n)+1)^k=|\mathcal{S}|^k\leq \prod_{i=1}^m |\mathcal{S}_i| \leq  x(r-u,n-u)+1)^m,$$
so that
$$\log(x(r,n)+1)\cdot\begin{bmatrix} r \\ u \end{bmatrix}_q\leq \log(x(r-u,n-u)+1)\cdot\begin{bmatrix} n \\ u \end{bmatrix}_q.$$
Using that $$\begin{bmatrix} n \\ r \end{bmatrix}_q/\begin{bmatrix} n-u \\ r-u \end{bmatrix}_q= \begin{bmatrix} n \\ u \end{bmatrix}_q/\begin{bmatrix} r \\ u \end{bmatrix}_q$$
the theorem follows.
\endproof

\begin{corollary} \label{coro:upperbound}
For all $1\leq r\leq n$, we have $$\log(s_q(r,n)+1)\leq \frac{1}{[n-r+1]_q}\qbinom{n}{r}_q \log\left([ n-r+1]_q+1\right).$$\end{corollary}
\proof A sparse paving $q$-matroid of rank 1 over $\F_q^t$ is determined by its flat of rank 0, which is a subspace of dimension at most 1. Hence 
$$s_q(1,t)=\begin{bmatrix} t \\ 1 \end{bmatrix}_q+\begin{bmatrix} t \\ 0 \end{bmatrix}_q=\begin{bmatrix} t\end{bmatrix}_q+1.$$ 
Taking $t=n-r+1$, and applying the theorem with $\cX=\mathcal{S}_q$ and $u=r-1$ yields the corollary.\endproof

The following corollary does not improve upon Theorem \ref{thm:nbrpaving}, but its proof demonstrates that an upper bound for rank 2 yields the same general bound via Theorem \ref{thm:entropybound}.  
\begin{corollary} \label{coro:upperbound2}
For all $2\leq r\leq n$, we have $$\log(p_q(r,n)+1)\leq \frac{1}{[n-r+1]_q}\qbinom{n}{r}_q \log\left(e[ n-r+1]_q\right).$$\end{corollary}
\proof By Proposition \ref{prop:reconstructcyclic}, a paving $q$-matroid $M$ of rank 2 over $E=\F_q^t$ is uniquely determined by a collection $\mathcal{V}(M)$ 2-subspaces  of $E$, and 
$$|\mathcal{V}(M)|\leq \frac{1}{[t-1]_q}\qbinom{t}{2}_q=:k$$
by Lemma \ref{lem:vmupperbound}. So each collection of 2-subspaces $\mathcal{V}$ of cardinality $\leq k$ yields at most one paving $q$-matroid, but not all such collections do: for example,  the collection of 2-subspaces $\cV$ as obtained from Lemma \ref{lem:H_encoding} for $H=E$ has $t-1\leq k$ elements, but there is no paving $q$-matroid $M$ so that $\cV(M)=\cV$ since the ground space $E$ is never a hyperplane.
Thus $$p_q(2,t)+1\leq \sum_{i=0}^k \binom{\qbinom{t}{2}_q}{i}\leq \left(e [t-1]_q\right)^k$$
Taking $t=n-r+2$, and applying the theorem with $\cX=\mathcal{P}_q$ and $u=r-2$ yields the corollary.\endproof


We will apply Theorem \ref{thm:entropybound} to bound the number of $q$-matroids along the same lines. To do so, we first bound the number of $q$-matroids of rank 2. We need the following technical lemma.

\begin{lemma}\label{lem:technical}
    Let $n \in \mathbb{N}$ such that $n \geq 2$,  and $1 \leq i \leq n-2$. Then 
    $$\qbinom{n}{i}_q(e[n-i-1]_q)^{\frac{[n-i]_q}{[2]_q}} \leq (e[n-1]_q)^{\frac{[n]_q}{[2]_q}}.$$
\end{lemma}

\begin{proof}
    To show the desired inequality we use the following facts. For $k \in \mathbb N$ we have
    \begin{equation}\label{bd:qbracket}
        q^{k-1} \leq [k]_q \leq q^k.
    \end{equation}
    For $n \in \mathbb N$ and $0 \leq i \leq n$, we have
    \begin{equation}\label{bd:qbinom}
        \qbinom{n}{i}_q \leq q^{n(i+1)-i^2}.
    \end{equation}
 Furthermore, for all $q \geq 2$ and $i \in \mathbb N$ we have $q^i \geq i$. Then
    \begin{align*}
        \frac{(e[n-1]_q)^{\frac{[n]_q}{[2]_q}}}{(e[n-i-1]_q)^{\frac{[n-i]_q}{[2]_q}}} & \geq \frac{(eq^{n-1})^{\frac{[n]_q}{[2]_q}}}{(eq^{n-i-1})^{\frac{[n-i]_q}{[2]_q}}} \\
        &\geq e^{\frac{[n]_q-[n-i]_q}{[2]_q}} q^{(n-1)\frac{[n]_q-[n-i]_q}{[2]_q} } q^{i\frac{[n-i]_q}{[2]_q}} \\
        &\geq e^{q^{n-3}}q^{(n-1)q^{n-3}}q^{iq^{n-i-3}}\\
        &\geq e^{q^{n-3}}q^{(n-1+i)q^{n-3} + iq^{-i}}\\
        &\geq q^{(n-3) +(n-1+i)q^{n-3} + iq^{-i}} \\
        &\geq q^{n(i+1) - i^2}\\
       & \geq \qbinom{n}{i}_q,
    \end{align*}
    as required.
\end{proof}

\begin{lemma}\label{lem:pav2} Let $t\geq 2$. Then
    $m_q(2,t)\leq (t-1) (e[t-1]_q)^{\frac{[t]_q}{[2]_q}}$.
\end{lemma}
\begin{proof}
    Consider $M=(E,\rho)\in \cM(2,t)$. Let $L\leq E$ be the loop space of $M$, that is, let  $L=\cl_M(\{0\})$ be the unique flat of rank 0 in $M$. Then $M$ is uniquely determined by the pair $(L, M/L)$, as $$\rho(V)=\rho_{M/L}(V/L)\qquad \text{ for all }V\leq E.$$ Moreover, $M/L$ is a paving $q$-matroid of size $\dim(E/L)=t-\dim(L)$. 
    For $i=0,\ldots, t-2$
    there are at most 
    $$\qbinom{t}{i}_qp_q(2, t-i)$$
    pairs $(L, M/L)$ so that $\dim(L)=i$. Ranging over $i=0, \ldots, t-2$, using Theorem \ref{thm:nbrpaving} to bound the number of paving $q$-matroids,  we find that 
    $$m_q(2,t)=\sum_{i=0}^{t-2}\qbinom{t}{i}_qp_q(2, t-i)\leq \sum_{i=0}^{t-2} \qbinom{t}{i}_q(e[t-i-1]_q)^{\frac{[t-i]_q}{[2]_q}}. $$
    Using Lemma \ref{lem:technical} to simplify each of the $t-1$ terms of the latter sum, the lemma follows.
\end{proof}

\begin{theorem}\label{thm:matupperbound}
    For all $2 \leq r \leq n$, we have 
 $$\log(m_q(r,n)+1)\leq \frac{1}{[n-r+1]_q}\qbinom{n}{r}_q\log(e[n-r+1]_q)(1+ \epsilon(r,n))$$
    where $\epsilon(r,n):=(\log(n-r+2)[2]_q)/([n-r+2]_q\log(e[n-r+1]_q))$. 
\end{theorem}
\begin{proof} 
Applying Theorem \ref{thm:entropybound} to $\cX=\cM_q$ with $u=r-2$, we obtain
$$\log(m_q(r,n)+1)\leq \frac{1}{\qbinom{n-r+2}{2}_q}\qbinom{n}{r}_q\log( m_q(2, n-r+2)+1).$$
Using Lemma \ref{lem:pav2}, we find that 
    \begin{align*}
        \log(m_q(2,t)+1) &\leq \log\left( (t-1) (e[t-1]_q)^{\frac{[t]_q}{[2]_q}} +1\right) \\
        &\leq \frac{[t]_q}{[2]_q}\log (e[t-1]_q) +\log(t).
         \end{align*}
    Substituting this bound on $\log(m_q(2,t)+1)$ with $t=n-r+2$, the theorem follows.
\end{proof}

\medskip

\section{Asymptotics}\label{sec:asymptoticcomparisson}

In this section we consider the asymptotic behaviour of the logarithmic lower and upper bounds obtained in this paper.
Let $$l_q(r,n):=3(n-r) + q^{(n-r)(r-1)}$$
and $$u_q(r,n):=\frac{1}{[n-r+1]_q}\qbinom{n}{r}_q\log(e[n-r+1]_q)(1+ \epsilon(r,n))$$
with $\epsilon(r,n):=([2]_q\log(n-r+2))/([n-r+2]_q\log(e[n-r+1]_q))$. 
Then by Proposition \ref{prop:lowerboundnumpaving} and Theorem \ref{thm:matupperbound}, we have
$$l_q(r,n)\leq \log(s_q(r,n))\leq \log(p_q(r,n))\leq \log(m_q(r,n))\leq u_q(r,n).$$

We write $f(t)\prec g(t)$ (resp. $f(t)\approx g(t)$) as $t\rightarrow\infty$ to denote that  
$$\lim_{t\rightarrow\infty} \frac{f(t)}{g(t)}=0\left(\text{resp.} \in \mathbb{R}_{>0}\right).$$
Then $\prec$ is a linear order on the nonnegative functions modulo the equivalence $\approx$.

Following a straightforward calculation, we find:
\begin{enumerate}
 \item If $2\leq r$ and $q$ is a prime power, then 
$$l_q(r,n)\approx q^{(n-r)(r-1)}\prec q^{(n-r)(r-1)}(n-r)\approx u_q(r,n)$$
 as $n\rightarrow \infty.$
 \item If $q$ is a prime power, then
$$l_q(r,2r)\approx q^{r(r-1)}\prec q^{r(r-1)}r\approx u_q(r, 2r)$$ as $r\rightarrow \infty$.
\item If $2\leq r\leq n$, then 
$$l_q(r,n)\approx q^{(n-r)(r-1)}\prec q^{(n-r)(r-1)}(n-r)\log(q)\approx u_q(r,n)$$
 as the prime powers $q\rightarrow \infty$. 
 \end{enumerate}
There is an asymptotic gap between our best upper and lower bound in all three regimes, which prompts the question which bound is closer to the asymptotic value of $\log(m_q(r,n))$.  

In the regime where $r$ is fixed and $n\rightarrow \infty$, there is a  lower bound on the number of (classical) sparse paving matroids which is close to the entropy upper bound on the number of matroids \cite{BPvdP2014, HPvdP2022}. We expect that a better lower bound on the number of sparse paving $q$-matroids can be derived along the same lines.
\begin{conjecture} For all $r\geq 2$ and prime powers $q$, we have
$$\log(s_q(r,n))\approx \log(p_q(r,n))\approx \log(m_q(r,n))\approx q^{(n-r)(r-1)}(n-r)$$
as $n\rightarrow \infty$.
\end{conjecture}
In the regime where $n=2r$ and $r\rightarrow \infty$, there is a small gap between the best upper and lower bounds on the number of sparse paving matroids, but a larger gap with the entropy upper bound on all matroids. It was nevertheless established in \cite{PvdP2016} that $\log(s(r,n))\approx\log(m(r,n))$. We believe that the methods of that paper should extend to $q$-matroids.
\begin{conjecture} For each prime power $q$, we have
$$q^{r(r-1)}\approx\log(s_q(r,2r))\approx\log(p_q(r,2r))\approx \log(m_q(r,2r))$$
as $r\rightarrow \infty$.
\end{conjecture}
There is no analogue in classical matroids for the asymptotics of $q\rightarrow\infty$.

\subsection*{Acknowledgements}
The authors are grateful to Aida Abiad for her stimulating and encouraging discussions about this project.
Benjamin Jany is supported by the Casimir Institute at the Eindhoven University of Technology.

\bibliographystyle{amsplain}
\bibliography{Bib}

\end{document}